\documentclass{amsart}
\usepackage{fresmb,amssymb}
\usepackage{graphicx}
\begin{document}
\title[Fredholm realizations]
{Fredholm realizations of elliptic symbols on manifolds with boundary II: fibered boundary}

\author{Pierre Albin}
\author{Richard Melrose}
\address{Department of Mathematics, Massachusetts Institute of Technology}
\email{pierre@math.mit.edu}
\email{rbm@math.mit.edu}

\thanks{
The first author was partially supported by an NSF postdoctoral fellowship and the second author received partial support under NSF grant DMS-0408993.}

\begin{abstract}
We consider two calculi of pseudodifferential operators on manifolds with fibered boundary: Mazzeo's edge calculus, which has as local model the operators associated to products of closed manifolds with asymptotically hyperbolic spaces, and the $\phi$ calculus of Mazzeo and the second author, which is similarly modeled on products of closed manifolds with asymptotically Euclidean spaces.
We construct an adiabatic calculus of operators interpolating between them, and use this to compute the `smooth' K-theory groups of the edge calculus, determine the existence of Fredholm quantizations of elliptic symbols, and establish a families index theorem in K-theory.
\end{abstract}

\maketitle

\section*{Introduction} \paperintro

If the boundary of a manifold is the total space of a fibration
\begin{equation}\label{BdyFibration}
	\xymatrix{
	Z \ar@{-}[r] & \pa X \ar[d]^{\Phi} \\ & Y }
\end{equation}
one can quantize an invertible symbol
\begin{equation*}
	\sigma \in \CI( S^*X; \hom(\pi^*E, \pi^*F) )
\end{equation*}
as an elliptic pseudodifferential operator in the $\fB{\Phi}$ or edge calculus, $\Psi^0_{\fB{\Phi}}(X;E,F)$, introduced in \cite{Mazzeo:Edge} 
or alternately as an elliptic operator in the $\fC{\Phi}$ or $\phi$ calculus, $\Psi^0_{\fC{\Phi}}(X;E,F)$, introduced in \cite{Mazzeo-Melrose1}.
As on a closed manifold, either of these operators will induce a bounded operator acting between natural $L^2$-spaces of sections but, in contrast to closed manifolds, these operators need not be Fredholm.

A well-known result of Atiyah and Bott \cite{Atiyah-Bott} established that for a differential operator on a manifold with boundary $X$ to admit local elliptic boundary conditions it is necessary and sufficient for the K-theory class of its symbol $[\sigma] \in K_c(T^*X)$ to map to zero under the natural map 
\begin{equation*}
	K_c(T^*X) \to K_c^1(T^*\pa X).
\end{equation*}
At one extreme, when $Z = \{\pt\}$, the $\fB{\Phi}$ calculus is the zero calculus of \cite{Mazzeo:Hodge} and the $\fC{\Phi}$ calculus is the scattering calculus of \cite{Melrose:Scat}, and for these the vanishing of the Atiyah-Bott obstruction is equivalent to the existence of a Fredholm quantization of $\sigma$.
At the other extreme, when $Y = \{\pt\}$, the $\fB{\Phi}$ calculus is the $b$-calculus of \cite{APSBook} and the $\fC{\Phi}$ calculus is the cusp calculus of \cite{Mazzeo-Melrose1}, and in either case there is no obstruction to finding a Fredholm realization of an elliptic symbol $\sigma$. 
The case of general $\phi$ is intermediate between these two extremes.
Indeed, it was established in \cite{Melrose-Rochon} that one can use $\Phi$ and the families index to induce a map $K_c(T^*X) \to K_c^1(T^*Y)$ (see \S\ref{sec:SixTerm}) and that $\sigma$ has a Fredholm quantization in the $\fC{\Phi}$ calculus if and only if
\begin{equation*}
	[\sigma] \in \ker \lrpar{ K_c(T^*X) \to K_c^1(T^*Y) }.
\end{equation*}
In this paper the corresponding result for the edge calculus is established.
We point out that another extension of the Atiyah-Bott obstruction, to a class of operators on stratified manifolds, is discussed in \cite{NSS}.

To this end we consider the even `smooth K-theory' group of the edge calculus $\cK_{\fB{\Phi}}(X)$ consisting of equivalence classes of Fredholm edge operators under the relations generated by bundle stabilization, bundle isomorphisms, and smooth homotopy (see \S \ref{Background}). This has a natural subgroup $\cK_{\fB{\Phi}, -\infty}(X)$ of equivalence classes with principal symbol the identity, and there are analogous odd smooth K-theory groups defined by suspension. 

The corresponding groups are defined and identified for the $\fC{\Phi}$ calculus in \cite{Melrose-Rochon}.
The analysis of these groups for the $\fC{\Phi}$ calculus is simpler for two related reasons. The first is that the $\fC{\Phi}$ calculus is `asymptotically normally commutative' whereas the $\fB{\Phi}$ calculus is `asymptotically normally non-commutative'. More precisely, the behavior of a $\fC{\Phi}$ operator near the boundary is modeled by a family of operators, parametrized by $Y$, acting on the Lie group $\bbR^{h+1}$ times the closed manifold $Z$. On the other hand, the behavior of a $\fB{\Phi}$ operator near the boundary is modeled by a family of operators, parametrized by $Y$, acting on the Lie group $\bbR^+ \ltimes \bbR^{h}$ times the closed manifold $Z$.
To see the effect of this difference on the analysis of these calculi, one can compare the relative simplicity of identifying $\cK_{\mathrm{sc}}(X)$ in \cite[\S 2]{Melrose-Rochon} versus the corresponding identification of $\cK_{0}(X)$ in \cite{Albin-Melrose}.

The other related simplification is that the $\fC{\Phi}$ calculus admits a smooth functional calculus, while the $\fB{\Phi}$ calculus does not. This means that for the $\fC{\Phi}$ calculus one can study the smooth K-theory groups using much the same constructions one would use to study the K-theory of its $C^*$-algebra (see for instance the constructions used in \cite[\S 4]{Melrose-Rochon} to define KK-classes).
One could remove this difficulty by passing to a $C^*$-closure of the $\fB{\Phi}$ calculus, but at the considerable cost of losing the smooth structure.

Instead, we will study the smooth K-theory groups of the $\fB{\Phi}$ calculus by constructing an `adiabatic' calculus of pseudodifferential operators interpolating between the $\fB{\Phi}$ and $\fC{\Phi}$ calculi. This induces maps, labeled $\ad$, between their smooth K-theory groups. 
We work more generally in the context of a fibration $X - M \xrightarrow{\phi} B$ where the fibers are manifolds with fibered boundaries, thus altogether we have
\begin{equation}\label{fullfib}
	\xymatrix { & & X \ar@{-}[r] & M \ar[d]^{\phi} \\
	Z \ar@{-}[r] & \pa X \ar[d]  \ar@{^{(}->}[ur] \ar@{-}[r] 
		& \pa M \ar[d]^{\Phi}  \ar@{^{(}->}[ur]  & B \\
	& Y \ar@{-}[r] & D \ar[ur] & }
\end{equation}
Using our analysis of the smooth K-theory groups of the zero calculus in \cite{Albin-Melrose}, we establish the following theorem.

\begin{theorem}\label{MainThm}
The maps
\begin{equation*}
	\cK_{\fC{\Phi}}(\phi) \xrightarrow{\ad} \cK_{\fB{\Phi}}(\phi), \quad
	\cK_{\fC{\Phi}}^1(\phi) \xrightarrow{\ad} \cK_{\fB{\Phi}}^1(\phi)
\end{equation*}
are isomorphisms, and restrict to isomorphisms 
\begin{equation*}
	\cK_{\fC{\Phi},-\infty}(\phi) \xrightarrow{\ad} \cK_{\fB{\Phi},-\infty}(\phi), \quad
	\cK_{\fC{\Phi},-\infty}^1(\phi) \xrightarrow{\ad} \cK_{\fB{\Phi},-\infty}^1(\phi).
\end{equation*}
Furthermore, these groups fit into a commutative diagram
\begin{equation*}
\xymatrix  @R=17pt @C=17pt {
\cK^0_{\fC{\Phi},-\infty}\lrpar{\phi} \ar[rr] \ar[rd]^{\ad} & & 
\cK^0_{\fC{\Phi}}\lrpar{\phi} \ar[rr] \ar[d]^{\ad} & & 
K^0_c\lrpar{T^*M/B} \ar@{->}[ddd] \ar@{<->}[ld]^= \\
& \cK^0_{\fB{\Phi},-\infty}\lrpar{\phi} \ar[r] & 
\cK^0_{\fB{\Phi}}\lrpar{\phi} \ar[r] &
K^0_c\lrpar{T^*M/B} \ar[d] & \\
& K^1_c\lrpar{T^*M/B} \ar[u] &
\cK^1_{\fB{\Phi}}\lrpar{\phi} \ar[l] &
\cK^1_{\fB{\Phi},-\infty}\lrpar{\phi} \ar[l] & \\
K^1_c\lrpar{T^*M} \ar@{->}[uuu] \ar@{<->}[ur]^= & &
\cK^1_{\fC{\Phi}}\lrpar{\phi} \ar[ll] \ar[u]^{\ad} & &
\cK^1_{\fC{\Phi},-\infty}\lrpar{\phi} \ar[ll] \ar[lu]^{\ad} \\}
\end{equation*}
wherein the inner and outer six term sequences are exact.
\end{theorem}

In \cite{Melrose-Rochon}, the groups $\cK_{\fC{\Phi}}^q(\phi)$ are identified in terms of the KK-theory of the $C(B)$-module
\begin{equation*}
	\cC_\Phi(M) = \{ f \in \cC(M) : f \rest{\pa M} \in \Phi^*\cC(D) \}
\end{equation*}
and the $\cK_{\fC{\Phi},-\infty}^q(\phi)$ groups are identified with the K-theory of the vertical cotangent bundle $T^*D/B$.
Below, we establish the analogue of the latter directly and deduce the analogue of the former from Theorem \ref{MainThm}. 

\begin{corollary}
There are natural isomorphisms
\begin{equation*}
	\cK_{\fB{\Phi}}^q(\phi) \cong KK_B^q( \cC_\Phi(M), \cC(B) ), \Mand
	\cK_{\fB{\Phi}, -\infty}^q(\phi) \cong K_c^q(T^*D/B).
\end{equation*}
\end{corollary}

From the six-term exact sequence and this corollary the topological obstruction to realizing an elliptic symbol via a family of Fredholm $\fB{\Phi}$ operators can be deduced.

\begin{corollary}
An elliptic symbol $\sigma \in \CI(S^*M/B; \hom(\pi^*E, \pi^*F) )$ can be quantized as a Fredholm family of $\fB{\Phi}$ operators if and only if its K-theory class satisfies
\begin{equation*}
	[\sigma] \in \ker \lrpar{ K_c(T^*M/B) \to K_c^1(T^*D/B) }.
\end{equation*}
\end{corollary}

Finally, in \cite{Melrose-Rochon} the second author and Fr\'ed\'eric Rochon defined a topological index map $\cK_{\fC{\Phi}}(\phi) \to K(B)$ and showed that it coincided with the analytic index.
Thus the group $\cK_{\fB{\Phi}}(\phi)$ inherits a topological index map from its isomorphism with $\cK_{\fC{\Phi}}(\phi)$ and, since the adiabatic limit commutes with the analytic index, the following K-theoretic index theorem follows.

\begin{corollary}
The analytic and topological indices coincide as maps 
\begin{equation*}
	\cK_{\fB{\Phi}} (\phi) \to K(B).
\end{equation*}
\end{corollary}

In Section \ref{Background}, we review the  $\fB{\Phi}$ and $\fC{\Phi}$ algebras of pseudodifferential operators and the definition of the smooth K-theory groups.
In Section \ref{AdCalc}, we set up the adiabatic calculus and prove that it is closed under composition in an appendix.
In Section \ref{sec:FredRes} we use an excision lemma and the analysis of the smooth K-theory of the zero calculus from \cite{Albin-Melrose} to identify the groups $\cK_{\fB{\Phi},-\infty}^q(\phi)$.
Finally, in Section \ref{sec:SixTerm}, we establish the six term exact sequence and finish the proof of Theorem \ref{MainThm}.

\paperbody
\section{$\fB{\Phi}$ and $\fC{\Phi}$ algebras of pseudodifferential operators} \label{Background}

We recall some of the main features of these calculi and refer the reader to, e.g., \cite{APSBook}, \cite{Mazzeo:Edge}, and \cite{Mazzeo-Melrose1} for more details.
For the moment we restrict attention to the case of a single operator (i.e., $B = \{\pt\}$).

We start by describing the vector fields that generate the differential operators in the two calculi.
Let $\{x, y_1, \ldots, y_h, z_1, \ldots, z_v \}$ be local coordinates near the boundary with $y_i$ lifted from the base under the fibration \eqref{BdyFibration} and the $z_i$ vertical. Here $x$ is a boundary defining function, i.e., a non-negative function on $\bar{X}$ with $\{x=0\} = \pa X$ and $dx \neq 0$ on the boundary. 
The fibered cusp structure depends mildly on this choice.

The Lie algebra, $\curly{V}_{\fB{\Phi}}$,  of vector fields tangent to the fibers of the fibration over the boundary is locally spanned by the vector fields
\begin{equation*}
	\{ x\pa_x, x\pa_{y_1},\ldots, x\pa_{y_h}, \pa_{z_1}, \ldots \pa_{z_v} \}.
\end{equation*}
Fibered boundary differential operators are polynomials in these vector fields. That is, any $P \in \mathrm{Diff}^k_{\fB{\Phi}}\lrpar{X}$ can be written locally as
\begin{equation*}
	P = \sum_{j+|\alpha|+|\beta| \leq k} 
		a_{j,\alpha, \beta}\lrpar{x,y,z} \lrpar{x\pa_x}^j\lrpar{x\pa_y}^\alpha\lrpar{\pa_z}^\beta.
\end{equation*}
There is a vector bundle ${}^{\fB{\Phi}}TX$, the $\fB{\Phi}$ tangent bundle, whose space of smooth sections is precisely $\curly{V}_{\fB{\Phi}}$. It plays the role of the usual tangent bundle in the study of the $\fB{\Phi}$ calculus. For instance, $\fB{\Phi}$ one-forms are elements of the dual bundle, ${}^{\fB{\Phi}}T^*X$, and a $\fB{\Phi}$ metric is a metric on ${}^{\fB{\Phi}}TX$, e.g. locally
\begin{equation}\label{PhibMetric}
	g_{\fB{\Phi}} = \frac{dx^2}{x^2} + \frac{\Phi^*g_Y}{x^2} + g_Z.
\end{equation}

`Extreme' cases of fibered boundary calculi are the b-calculus (where $Y$ is a point) and the $0$-calculus (where $Z$ is a point). The b-calculus models non-compact manifolds with a cylindrical end. It was used in \cite{APSBook} to prove the APS-index theorem. The $0$-calculus models non-compact manifolds that are asymptotically hyperbolic. It has applications to conformal geometry through the Fefferman-Graham construction and to physics (e.g. holography) through the AdS/CFT correspondence \cite{Graham}.

There is also a $\fC{\Phi}$ tangent bundle, whose space of sections $\curly{V}_{\fC{\Phi}}$ is locally spanned by
\begin{equation*}
	\{ x^2\pa_x, x\pa_{y_1},\ldots, x\pa_{y_h}, \pa_{z_1}, \ldots \pa_{z_v} \}.
\end{equation*}
Thus a fibered cusp differential operator
$P \in \mathrm{Diff}^k_{\fC{\Phi}}\lrpar{X}$ can be written locally as
\begin{equation*}
	P = \sum_{j+|\alpha|+|\beta| \leq m} 
		a_{j,\alpha, \beta}\lrpar{x,y,z} \lrpar{x^2\pa_x}^j\lrpar{x\pa_y}^\alpha\lrpar{\pa_z}^\beta.
\end{equation*}

The `extreme' cases of fibered cusp calculi are known as the cusp-calculus (where $Y$ is a point) and the scattering-calculus (where $Z$ is a point). 
The cusp-calculus models the same geometric situation (asymptotically cylindrical manifolds) as the b-calculus, and indeed these calculi are very closely related (see \cite[\S 3]{Albin-Melrose}).
The scattering-calculus models non-compact, asymptotically locally Euclidean manifolds. Indeed, if one compactifies $\RR^n$ radially to a half-sphere, the metric near the boundary takes the form
\begin{equation*}
	\frac{dx^2}{x^4} + \frac{h_x}{x^2},
\end{equation*}
and so defines a metric on the scattering tangent bundle.

As an illustration of the difference between the $\fB{\Phi}$ and $\fC{\Phi}$ calculi we point out that the scattering Lie algebra is asymptotically commutative in the sense that
\begin{equation*}
	\lrspar{\curly{V}_{sc}, \curly{V}_{sc} } \subset x\curly{V}_{sc},
\end{equation*}
while the sub-bundle of the zero tangent bundle ${}^0 TX$ spanned by commutators of zero vector fields is non-trivial over the boundary.
This asymptotic commutativity of horizontal vector fields in the $\fC{\Phi}$ calculus lies behind the simplifications over the $\fB{\Phi}$ calculus.

As on a closed manifold, certain interesting operations (e.g. powers, parametrices, or inverses) require passing to a larger calculus of pseudodifferential operators. Pseudodifferential operators mapping sections of a bundle $E$ to sections of a bundle $F$ are denoted $\Psi^*_{\fB{\Phi}}\lrpar{X;E,F}$ and $\Psi^*_{\fC{\Phi}}\lrpar{X;E,F}$ respectively. Operators in these calculi act by means of distributional integral kernels as in the Schwartz kernel theorem, i.e.
\begin{equation*}
	Pf\lrpar{\zeta} = \int_X \cK_{P}\lrpar{\zeta, \zeta'}f\lrpar{\zeta'}.
\end{equation*}
These integral kernels have singularities along the diagonal and when $\zeta, \zeta' \in \pa X$. The latter can be resolved by lifting the kernel to an appropriate blown-up space, denoted respectively $X^2_{\fB{\Phi}}$ and $X^2_{\fC{\Phi}}$. We will describe these spaces below, in $\S$\ref{AdCalc}, as part of the construction of the adiabatic calculus.

These pseudodifferential calculi each possess two symbol maps. On a closed manifold, the highest order terms in the local expression of a differential operator define invariantly a function on the cosphere bundle of the manifold. This same construction yields the principal symbol of a differential (and more generally pseudodifferential) $\fB{\Phi}$ operator as a function on the $\fB{\Phi}$ cosphere bundle, i.e., the bundle of unit vectors in ${}^{\fB{\Phi}}T^*X$.
This interior symbol fits into the short exact sequence,
\begin{equation*}
	\Psi^{k-1}_{\fB{\Phi}}\lrpar{X;E,F} \hookrightarrow
	\Psi^k_{\fB{\Phi}}\lrpar{X;E,F} 
	\substack{\sigma \\ \twoheadlongrightarrow \\ \phantom{\sigma}}
	C^{\infty}\lrpar{ {}^{\fB{\Phi}}S^*X; \lrpar{N^*X}^k \otimes \hom\lrpar{\pi^*E,\pi^*F}},
\end{equation*}
where $\lrpar{N^*X}^k$ denotes a line bundle whose sections have homogeneity $k$ and $\pi$ is the projection map ${}^{\fB{\Phi}}S^*X \to X$.
The interior symbol in the fibered cusp calculus works in much the same way, and the corresponding sequence 
\begin{equation*}
	\Psi^{k-1}_{\fC{\Phi}}\lrpar{X;E,F} \hookrightarrow
	\Psi^k_{\fC{\Phi}}\lrpar{X;E,F} 
	\substack{\sigma \\ \twoheadlongrightarrow \\ \phantom{\sigma}}
	C^{\infty}\lrpar{ {}^{\fC{\Phi}}S^*X; \lrpar{N^*X}^k \otimes \hom\lrpar{\pi^*E,\pi^*F}},
\end{equation*}
is also exact.
The interior symbol is used to define ellipticity: an operator is elliptic if and only if its interior symbol is invertible.

The second symbol map, known as the normal operator, models the behavior of the operator near the boundary. Following \cite{EpsteinMelroseMendoza}, at any point $p \in Y$ let $\curly{I}_{\Phi^{-1}\lrpar{p}}\subset\curly{V}_{\fB{\Phi}}$ be the subspace of $\fB{\Phi}$ vector fields that vanish along the fiber $\Phi^{-1}\lrpar{p}$ (as $\fB{\Phi}$ vector fields). These vector fields form an ideal and the quotient ${}^{\fB{\Phi}}T_pX$ is a Lie algebra. 
The projection 
\begin{equation*}
	\curly{V}_{\fB{\Phi}} \to {}^{\fB{\Phi}}T_pX
\end{equation*}
lifts to a map of enveloping algebras
\begin{equation*}
	\mathrm{Diff}^k_{\fB{\Phi}}\lrpar{X} \to \curly{D}\lrpar{{}^{\fB{\Phi}}T_pX},
\end{equation*}
and the image of $P\in\mathrm{Diff}^k_{\fB{\Phi}}\lrpar{X}$ is known as the normal operator of $P$ at $p$, $N_{\fB{\Phi},p}\lrpar{P}$.
The normal operator extends to pseudodifferential operators and can be realized either globally (its kernel is obtained by restricting the kernel of $P$ to a certain boundary face in $X^2_{\fB{\Phi}}$) or locally by means of an appropriate rescaling as follows.

Assume that we have chosen a product neighborhood of the boundary and, in a neighborhood $\curly{U}_Y$ of a point $p \in Y$, we have chosen a local trivialization of $\Phi$, so that $\curly{U} \subset X$ looks like
\begin{equation}\label{ProdDecomp}
	\curly{U} \cong \left[ 0, \eps \right) \times \curly{U}_Y \times Z  \xrightarrow{\Phi} \curly{U}_Y,
\end{equation}
with coordinates $\{x\}$, $\{y_i\}$, and $\{z_i\}$ on the respective factors, and a corresponding decomposition of the tangent bundle
\begin{equation*}
	{}^{\fB{\Phi}}TX = {}^{\fB{\Phi}}N\pa X \oplus V\pa X.
\end{equation*}
Using \eqref{ProdDecomp} we can define the dilation
\begin{equation*}
	\lrpar{x,y_i,z_j} \overset{M_\delta}{\mapsto} \lrpar{\delta x, \delta y_i, z_j}.
\end{equation*}
If $V \in \curly{V}_{\fB{\Phi}}$ then in local coordinates
\begin{equation*}
	V_\delta = \lrpar{M_\delta^{-1}}^*V
\end{equation*}
is a smooth vector field defined in a neighborhood of zero (increasing as $\delta$ decreases). The map
\begin{equation*}
	\cV_{\fB{\Phi}} \ni V \mapsto \lim_{\delta \to 0} V_\delta \in T({}^{\fB{\Phi}} N \pa Z) \oplus V\pa Z
\end{equation*}
has kernel $\curly{I}_{p,\fB{\Phi}}$ and realizes ${}^{\fB{\Phi}}T_pZ$ as a Lie algebra of smooth vector fields on $T_pZ$.
More generally, Let $L_{\curly{U}}$ be a chart at $p$ mapping into $N^+_pY$, and define the normal operator of $P\in \Psi^0_{\fB{\Phi}}\lrpar{X}$ at $p$ to be
\begin{equation*}
	N_{p,\fB{\Phi}}\lrpar{P}u = 
	\lim_{\delta \to 0} \lrpar{M_\delta}^*L_\curly{U}^*P\lrpar{L_\curly{U}^{-1}}^*\lrpar{M_{\frac1\delta}}^*u.
\end{equation*}
The Lie group structure on ${}^{\fB{\Phi}}N\pa X$ at the point $p \in Y$ is that of $\RR^+ \ltimes \RR^h$, i.e.
\begin{equation*}
	\lrpar{s,u}\cdot \lrpar{s',u'}
	= \lrpar{ss', u+su'}.
\end{equation*}
The kernel of the normal operator at the point $p \in Y$ is invariant with respect to this action, so the normal operator acts by convolution in the associated variables.
We refer to this as a non-commutative suspension, and denote these operators by $\Psi^*_{N\mathrm{\sus}}\lrpar{{}^{\fB{\Phi}}N\pa X; E, F}$. They fit into a short exact sequence
\begin{equation*}
	x\Psi^k_{\fC{\Phi}}\lrpar{X;E,F} \hookrightarrow
	\Psi^k_{\fC{\Phi}}\lrpar{X;E,F} 
	\substack{ N_{\fB{\Phi}} \\ \twoheadlonglongrightarrow \\ \phantom{ N_{\fB{\Phi}} }}
	\Psi^k_{N\mathrm{\sus}}\lrpar{{}^{\fB{\Phi}}N\pa X; E, F}.
\end{equation*}

The same construction yields a model operator at every point $p \in Y$ for the fibered cusp calculus. In this case the Lie group is commutative and isomorphic to $\RR \times \RR^h$; the resulting calculus is referred to as the suspended calculus. The corresponding short exact sequence is
\begin{equation*}
	x\Psi^k_{\fC{\Phi}}\lrpar{X;E,F} \hookrightarrow
	\Psi^k_{\fC{\Phi}}\lrpar{X;E,F} 
	\substack{ N_{\fC{\Phi}} \\ \twoheadlonglongrightarrow \\ \phantom{ N_{\fC{\Phi}} }}
	\Psi^k_{\mathrm{\sus}}\lrpar{{}^{\fC{\Phi}}N\pa X; E, F}.
\end{equation*}

The volume form of a $\fB{\Phi}$ metric (e.g. \eqref{PhibMetric}), together with Hermitian metrics on $E$ and $F$, defines a space of $L^2$ sections. An operator $P \in \Psi^k_{\fB{\Phi}}\lrpar{X;E,F}$ acts linearly on $L^2$ sections of $E$; boundedly if $k \leq 0$. There is a corresponding scale of $\fB{\Phi}$ Sobolev spaces, $H_{\fB{\Phi}}^s\lrpar{X;E}$, and 
an element $P \in \Psi^k_{\fB{\Phi}}\lrpar{X;E,F}$ defines a bounded linear operator
\begin{equation*}
	H_{\fB{\Phi}}^s\lrpar{X;E} \xrightarrow{P} H^{s-k}_{\fB{\Phi}}\lrpar{X;F}.
\end{equation*}
This operator is Fredholm if and only if both symbol maps $\sigma\lrpar{P}$ and $N_{\fB{\Phi}}\lrpar{P}$ are invertible operators, in which case $P$ is said to be {\em fully elliptic}.

In order to carry out our analysis below, we will need a good understanding of the normal operators of elements of $\Psi^{-\infty}_0$. A construction from \cite{Mazzeo:Edge} and \cite{Lauter} realizes normal operators of the zero calculus as families of `b,c-calculus' on the interval, referred to as {\em reduced normal operators}. Operators in this calculus have b-behavior at one end of the interval and cusp-behavior at the other end. The reduced normal operator of an operator in $\Psi^{-\infty}_0\lrpar{X;E}$ lies inside the space
\begin{equation*}
	\Psi_{b,c}^{-\infty,-\infty}\lrpar{\lrspar{0,1},\pi^*E}
\end{equation*}
of operators with b-behavior of order $-\infty$ near $0$ and which vanish to infinite order near $1$. In the following result from \cite{Lauter}, $\pi$ denotes the canonical projection $S^*\pa X \to \pa X$.

\begin{proposition}[\cite{Lauter}, Prop. 4.4.1] \label{CharacRes}
The reduced normal operators of elements of $\Psi^{-\infty}_0\lrpar{X;E}$ are precisely those functions 
\begin{equation}\label{RedNRes} \begin{split}
	\curly{N}\lrpar{y', \eta; \tau, \rho} 
	&\in C^{\infty}\lrpar{S^*\partial X} 
	\wh{\otimes}_{\pi} \dot{C}^{\infty}\lrpar{\lrspar{-1,1}, \Omega^{1/2}}
	\wh{\otimes}_{\pi} \Sc\lrpar{\RR_+, \half{b}}\\
	&\phantom{xxx} \otimes_{C^{\infty}\lrpar{\wt{\curly{I}}^2}}
	C^{\infty}\lrpar{\wt{\curly{I}}^2, \beta^*\lrpar{\Hom\lrpar{\pi^*E \otimes \Half{b,c}}}}
\end{split}\end{equation}
with $\curly{N}\lrpar{y', \eta; \cdot, \rho}$ the lift of a density on $T^*\partial X$, 
and which for each $\eta \in S^*\partial X$ extend to an element of 
$\Psi_{b,c}^{-\infty,-\infty}\lrpar{\lrspar{0,1},\pi^*E}$ in such a way that
\begin{equation}\label{FunSpace}
	\curly{N} \in C^{\infty}\lrpar{S^*\partial X, \Psi_{b,c}^{-\infty,-\infty}\lrpar{\lrspar{0,1},\pi^*E}}.
\end{equation}
\end{proposition}

{\em Remark.} Notice that the fact that 
$\curly{N}\lrpar{y', \eta; \cdot, \rho}$ is the lift of a density on 
$T^*\partial X$ implies in particular that the b-normal operator is a family defined on $S^*\pa X$ but actually only depending on $\pa X$.

Next, we define K-theory groups of $\fB{\Phi}$ operators. 
Following \cite[Definition 2]{Melrose-Rochon}, define 
\begin{equation*}
A_{\fB{\Phi}}\lrpar{M;E, F} 
= \{ \lrpar{\sigma\lrpar{A}, \curly{N}\lrpar{A}}: A\in \Psi^0_{\fB{\Phi}}(M;E, F) \text{ Fredholm} \}
\end{equation*}
so that $\cK^0_{\fB{\Phi}}\lrpar{M}$ consists of equivalence classes of elements in $A_{\fB{\Phi}}\lrpar{M;E, F}$, where two elements are equivalent if there is a finite chain consisting of the following:
\begin{align}
	\lrpar{\sigma,N} \in A_{\fB{\Phi}}\lrpar{M;E, F} &\sim \lrpar{\sigma',N'} 
		\in A_{\fB{\Phi}}\lrpar{M;E', F'} \label{Equiv1}\\
	\intertext{if there exists a bundle isomorphisms $\gamma_E:E \to E'$ and $\gamma_F:F \to F'$ 
	such that $\sigma = \gamma_F^{-1} \circ \sigma' \circ \gamma_E$
	and $\cN = \gamma_F^{-1} \circ \cN' \circ \gamma_E$,}
	\lrpar{\sigma,N} \in A_{\fB{\Phi}}\lrpar{M;E, F} &\sim %
	\lrpar{\wt\sigma,\wt{N}} \in A_{\fB{\Phi}}\lrpar{M;E, F} \label{Equiv2}\\
	\intertext{if there exists a homotopy of Fredholm operators $A_t$ in 
		$\Psi_{\fB{\Phi}}^0\lrpar{M;E, F}$, with
	$\lrpar{\sigma,N} = \lrpar{\sigma\lrpar{A_0},\curly{N}\lrpar{A_0}}$ and
	$\lrpar{\wt{\sigma},\wt{N}}=\lrpar{\sigma\lrpar{A_1},\curly{N}\lrpar{A_1}}$, and}
	\lrpar{\sigma,N} \in A_{\fB{\Phi}}\lrpar{M;E, F} &\sim %
	\lrpar{\sigma\oplus \Id_G, N \oplus \Id_G} \in A_{\fB{\Phi}}\lrpar{M;E\oplus G, F \oplus G}. \label{Equiv3}
\end{align}
Similarly,
we define $\cK^1_{\fB{\Phi}}\lrpar{M}$ as equivalence classes of elements in the space of based loops
\begin{equation*}
	\Omega A_{\fB{\Phi}}\lrpar{M;E, F}
	= \{ s \in C^{\infty}\lrpar{\mathbb{S}^1,A_{\fB{\Phi}}\lrpar{M;E, F}} :  s\lrpar{1} = \Id \},
\end{equation*}
where the equivalences are finite chains of \eqref{Equiv1}, \eqref{Equiv2}, \eqref{Equiv3} with bundle transformations and homotopies required to be the identity at $1\in\mathbb{S}^1$.

In the same way, we can describe $\cK_{\fB{\Phi},-\infty}^0\lrpar{M}$ and $\cK_{\fB{\Phi},-\infty}^1\lrpar{M}$ as equivalence classes of elements in 
\begin{equation*}
	A_{\fB{\Phi},-\infty}\lrpar{M;E, F} = 
	\{ N\lrpar{\Id + A} : A \in \Psi_{\fB{\Phi}}^{-\infty}\lrpar{M;E, F}, 
		\Id + A \phantom{x}\mathrm{Fredholm} \},
\end{equation*}
and
\begin{equation*}
	\Omega A_{\fB{\Phi},-\infty}\lrpar{M;E, F}
	= \{ s \in C^{\infty}\lrpar{\mathbb{S}^1,A_{\fB{\Phi},-\infty}\lrpar{M;E, F}} :  s\lrpar{1} = \Id \}
\end{equation*}
respectively.

\section{Adiabatic limit of edge to $\phi$ calculi} \label{AdCalc}

For a compact manifold with boundary, with a specified fibration of the boundary, we show that there is an adiabatic limit construction passing from the fibered boundary (for $\eps >0$) calculus to the fibered cusp calculus in the limit. Every Fredholm (i.e., totally elliptic) operator in the limiting calculi occurs in a totally elliptic family in the adiabatic calculus, so with constant index. This allows the K-theory of the two algebras to be identified and the (families) index map for one to be reduced to that of the other.

Let $X$ be a compact manifold with boundary, with boundary fibration \newline
$\xymatrix @C=1pc { Z \ar@{-}[r] & \pa X \ar[r] & Y}$. 
We construct a resolution of $X^2 \times \lrspar{0,\eps_0}$ which carries the adiabatic calculus. First blow up the corner and consider
\begin{equation*}
	X^2_{\lrpar{1}} = \lrspar{ X^2 \times \lrspar{0,\eps_0}; \lrpar{\pa X}^2 \times \{ 0\}}
	\xrightarrow{\beta_{\lrpar{1}}} X^2 \times \lrspar{0, \eps_0}.
\end{equation*}
Then blow up the two sides to get
\begin{equation*}
	X^2_{\lrpar{2}} = \lrspar{ X^2_{\lrpar{1}}; X \times \pa X \times \{ 0 \}; \pa X \times X \times \{0\}}
	\xrightarrow{\beta_{\lrpar{2}}} X^2_{\lrpar{1}}.
\end{equation*}
The lifts of these manifolds are disjoint in $X^2_{\lrpar{1}}$ and commutativity of blow-ups shows that there is a smooth map, which is in fact a b-fibration, to the rescaled single space 
\begin{equation*}
	X_\eps = \lrspar{ X \times \lrspar{0,\eps_0}; \pa X \times \{0\}}
\end{equation*}
in either factor:
\begin{equation*}
	X^2_{\lrpar{2}} \underset{\pi_{R,\eps}}{\overset{\pi_{L,\eps}}{\rightrightarrows}} 
	X_\eps.
\end{equation*}
This leads to the action of the adiabatic operators on $\CI\lrpar{X_{\eps}}$.

The choice of the fibered-cusp structure on $X$ corresponds to the singling out of a class of boundary defining functions which all induce the same trivialization of the normal bundle to the boundary along the fibres of $\Phi$.
Let $x$ be such a boundary defining function, on the left factor of $X$ in $X^2$ and let $x'$ denote the same function on the right factor. The hypersurface $x=x'$ is smooth near $\lrpar{ \pa X}^2 \subset X^2$ and lifts to be smooth near the front face of $X^2_{\lrpar{2}}$, $\ff\lrpar{X^2_{\lrpar{2}}}$, i.e. the lift of $\lrpar{\pa X}^2 \times \{0\}$. In fact it meets the front face in a smooth hypersurface; we denote by $\df{F}$ its intersection with the fibre diagonal
\begin{equation*}
	\df{F} \subset \ff\lrpar{X^2_{\lrpar{2}}}.
\end{equation*}
In fact it is already smooth in $X^2_{\lrpar{1}}$ and does not meet the side faces blown up to produce $X^2_{\lrpar{2}}$.

Note that $\df{F}$ is not a $p$-submanifold because of its intersection with the lift of $\lrpar{\pa X}^2 \times \lrspar{0,\eps_0}$ where it meets the boundary of the fibre diagonal
\begin{equation*}
	\df{B} \subset X^2_{\lrpar{2}}
\end{equation*}
really of course the lift of the fibre diagonal inside $\lrpar{\pa X}^2 \times \lrspar{0,\eps_0}$ to $X^2_{\lrpar{2}}$. 
The blow up of $\df{B}$ resolves $\df{F}$ to a $p$-submanifold, so we can set
\begin{equation*}
	X^2_{\eps \Phi} = \lrspar{X^2_{\lrpar{2}};\df{B};\df{F}}.
\end{equation*}
We will also use $\df B$ and $\df F$ to denote the boundary hypersurfaces of $X^2_{\eps\Phi}$ that result from blowing up these submanifolds.

\begin{figure}[h]
	\begin{minipage}{2in}
		\center{\includegraphics[scale=.95]{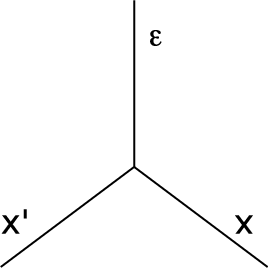}}{\caption{$X^2 \times [0,\eps_0]$}}
	\end{minipage}
	\begin{minipage}{2in}
		\center{\includegraphics[scale=.66]{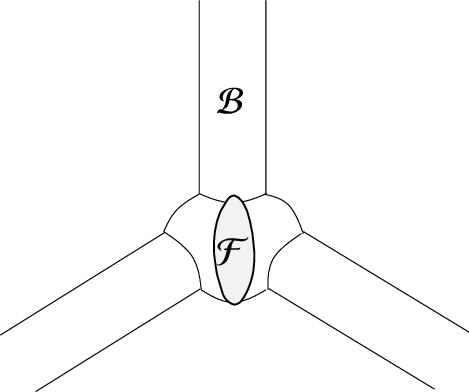}}{\caption{$X^2_{\eps \Phi}$}}
	\end{minipage}
\end{figure}

\begin{proposition}\label{DoubleAdSpace}
The diagonal, $\diag_{\eps\Phi} = \diag \times [0,\eps_0] \subset X^2 \times [0,\eps_0]$, lifts to an interior $p$-submanifold of $X^2_{\eps\Phi}$ and the projections lift to b-fibrations
 \begin{equation*}
	X^2_{\eps\Phi} \underset{\pi_{R}}{\overset{\pi_{L}}{\rightrightarrows}} 
	X_\eps
\end{equation*}
which are transverse to the lifted diagonal.
\end{proposition}

\begin{proof}
To show that the diagonal lifts to an interior $p$-submanifold, it is only necessary to consider it near the new boundary faces.
Let $\{ x, y_1, \ldots, y_h, z_1, \ldots, z_k \}$ be local coordinates near the boundary on the left factor of $X^2$ as in \S\ref{Background}, and $\{ x',y_1', \ldots, y_h', z_1', \ldots, z_k' \}$ a corresponding set of coordinates on the right factor of $X^2$.
Then local coordinates on $X^2_{\eps\Phi}$ near $\df B \cap \diag_{\eps\Phi}$ are given by
\begin{equation*}
	s = \frac x{x'}, x', u_i = \frac{y_i - y_i'}{x'}, y_i',  z_j, z_j' .
\end{equation*}
In these coordinates, $\diag_{\eps\Phi} = \{ s=1, u_i = 0, z_j = z_j' \}$, hence $\diag_{\eps\Phi}$ is a $p$-submanifold near $\df B$.
In the corresponding local coordinates near $\df F \cap \diag_{\eps\Phi}$,
\begin{equation*}
	S = \frac{s-1}{x'}, x', u_i, y_i' , z_j, z_j',
\end{equation*}
the diagonal is given by $\diag_{\eps\Phi} = \{ S = 0, u_i = 0, z_j = z_j' \}$ and is again a $p$-submanifold.
Next, to see that the maps $\pi_R$ and $\pi_L$ lift to b-fibrations, note that they do not map any boundary hypersurface to a corner and are fibrations in the interior of each boundary face.
Finally, they are transverse to the lifted diagonal from the local coordinate description of $\diag_{\eps\Phi}$.
\end{proof}

For any $\bbZ_2$-graded vector bundle $\bbE = (E_+, E_-)$ over $X$ the space
$\Psi^k_{\eps\Phi}\lrpar{X;\bbE}$ 
of adiabatic-$\Phi$ pseudodifferential operators on $X$, between sections of $E_+$ and sections of $E_-$, is defined to be 
\begin{multline*}
	\Psi^k_{\eps\Phi}\lrpar{X;E_+, E_-}
	= \bigl\{ A \in I^{k-\frac14}
	\lrpar{X^2_{\eps\Phi}, \mathrm{Diag} \times \lrspar{0,\eps_0};%
		\Hom\lrpar{E_+,E_-}\otimes\Omega_{\eps\Phi}}: \\
	A \equiv 0  
	\text{ at the lifts of } 
	\pa X \times X \times \lrspar{0,\eps_0},
	X \times \pa X \times \lrspar{0,\eps_0},
	\pa X \times X \times \{0\},
	X \times \pa X \times \{0\}  \bigr\} .
\end{multline*}

\begin{proposition}
These kernels define continuous linear operators
\begin{equation*}
	\CI\lrpar{X_\eps;E_+} \to
	\CI\lrpar{X_\eps;E_-}.
\end{equation*}
\end{proposition}

\begin{proof}
This follows from the push-forward and pull-back theorems of \cite{Corners} together with proposition \ref{DoubleAdSpace}. 
Indeed, given a section $f\in \CI\lrpar{X_\eps,{}^{\eps\Phi}\Omega^{1/2}}$ and an operator $A \in 	\Psi^k_{\eps\Phi}\lrpar{X;{}^{\eps\Phi}\Omega^{1/2}}$ we have
\begin{equation*}
	Af = \lrpar{\pi_L}_*\lrpar{A\cdot\pi_R^*f}.
\end{equation*}
Since $A$ vanishes to infinite order at all faces of $X^2_{\eps\Phi}$ not meeting the diagonal, and each boundary face meeting the diagonal maps down to a unique boundary face in $X_\eps$ it follows directly that smooth functions are mapped to smooth functions. In the same way, if $\curly{E}$ is a smooth index set for $X_\eps$ and $\curly{A}^{\curly{E}}$ is the associated set of polyhomogeneous functions then
\begin{equation*}
	\curly{A}^{\curly{E}} \xrightarrow{A}
	\curly{A}^{\curly{E}}.
\end{equation*}
The same argument holds with bundle coefficients with only notational differences.
\end{proof}

We will give a similar geometric proof of the composition
\begin{equation*}
	\Psi^k_{\eps\Phi}\lrpar{X;G,H} \circ 
	\Psi^{k'}_{\eps\Phi}\lrpar{X;E,G} \subset
	\Psi^{k+k'}_{\eps\Phi}\lrpar{X;E,H}.
\end{equation*}
in the appendix (Proposition \ref{AdComposition}), after we describe the corresponding triple space.

Operators in the adiabatic calculus, along with the usual interior symbol, have four `normal operators' from restricting their kernels to the different boundary faces.
These restrictions in turn have natural interpretations as operators in simpler calculi derived from the 
suspended and non-commutative suspended calculi in which the 
normal operators of the fibered cusp and fibered boundary calculi respectively lie.
All together, we have five surjective homomorphisms
\begin{equation*}
\begin{pmatrix}
\sigma \\
N_{\df{B}}\\
N_{\df{F}}\\
R_0\\
R_1
\end{pmatrix}
:\Psi^{m,k}_{\aD{\Phi}}(X;\bbE)
\longrightarrow
\begin{pmatrix}
S^{m}({}^{\aD{\Phi}}T^*X,\hom(\bbE))\\
\CI\lrpar{ \lrspar{0,1}_\eps, \Psi^k_{\Phi\text{-Nsus}}\lrpar{X;\bbE} } \\
\CI\lrpar{ \lrspar{-1,1}_\tau, \Psi^k_{\Phi\text{-sus}}\lrpar{X;\bbE} }\\
\Psi^m_{\fC{\Phi}}(X;\bbE)\\
\Psi^m_{\fB{\Phi}}(X;\bbE)\end{pmatrix}
\end{equation*}
with range respectively the homogeneous sections, 
the non-commutative $\Phi$-suspended calculus, 
the $\Phi$-suspended calculus,  
the fibered cusp and fibered boundary calculi respectively.

The model adiabatic calculus, where $N_{\df{F}}$ takes values, is a bundle of suspended pseudodifferential algebras on $Y \times \lrpar{0,1}$. 
Namely, if $\ff(X_{\eps})$ denotes the lift of $\pa X \times \{0\}$ to $X_\eps$, then the adiabatic front face of $X^2_{\eps\Phi}$ as a bundle over $Y\times \lrspar{-1,1} = \ff\lrpar{X_{\eps}}/Z$ (where  $Z$ is the fiber of $\Phi$) with fiber over 
$\{y\} \times \{\frac{x-\eps}{x+\eps} \}$ the product of $RC(W) \times Z_y \times Z_y$ with $RC(W)$ the radial compactification of the subspace 
\begin{equation*}
	W = \ker\lrpar{ {}^{\eps\Phi} TX_\eps \hookrightarrow TX_\eps}.
\end{equation*}

It is useful to see how the product on the fibers evolves as $\eps \to 0$. 
For simplicity assume we are in the case $Z = \{\pt\}$. 
In projective coordinates away from $\eps = 0$, 
\begin{equation*}
	\lrpar{x,y,s,u, \eps}= \lrpar{x,y,\frac{x'}x, \frac{y'-y}x, \eps},
\end{equation*}
the product on the fiber of the zero front face over the point $\lrpar{0,y_0, \eps}$ is given by
\begin{equation*}
	\lrpar{s,u}\cdot \lrpar{s',u'} = \lrpar{ss', u+su'}.
\end{equation*}
In the perhaps less natural coordinates
\begin{equation*}
	\lrpar{x,y,S,U,\eps} = \lrpar{x,y,\frac{x'-x}x, \frac{y'-y}x, \eps} 
\end{equation*}
this product takes the form
\begin{equation*}
	\lrpar{S,U}\cdot \lrpar{S',U'} = \lrpar{S+ S'+SS', U+U'+SU'}.	
\end{equation*}
Near the adiabatic front face we can use coordinates
\begin{equation*}
	\lrpar{x,y,\eps, \tau, S_\eps, U_\eps } 
	= \lrpar{x,y, \eps, \frac{x-\eps}{x+\eps}, \frac{x'-x}{x\lrpar{x+\eps}}, \frac{y'-y}x}
\end{equation*}
and obtain a family of products over the point $\lrpar{0,y_0,0,\tau_0}$
\begin{equation*}
	\lrpar{S_\eps,U}\cdot \lrpar{S'_\eps,U'} 
	= \lrpar{S_\eps+ S'_\eps+\eps S_\eps S'_\eps, U+U'+\eps S_\eps U'}.
\end{equation*}
Note that if $\eps = 1$ this product coincides with that on the zero front face, while at the adiabatic front face ($\eps = 0$) it coincides with the product on the scattering front face.

Note that this construction extends to families: If $X - M \xrightarrow{\phi} B$ is a fibration with $X$ as above, then carrying out the above blow-ups on $M \ftimes_B M$ (the fiber product of $M$ with itself) we obtain a fibration
\begin{equation*}
	\xymatrix{ X^2_{\eps\Phi} \ar@{-}[r] & M^2_{\eps\Phi} \ar[d] \\ & B}
\end{equation*}
which carries the kernels of families of adiabatic operators, $\Psi^*_{\eps\Phi}\lrpar{M/B;\bbE}$.

\begin{proposition}[cf. \cite{Melrose-Rochon}, Proposition C.1] \label{SameIndex}
If $P\in \Psi^m_{\eps\Phi}\lrpar{M/B;\EE}$ and the normal operators 
$\sigma\lrpar{P}$, $N_{\df{B}}\lrpar{P}$ and $N_{\df{F}}\lrpar{P}$ are invertible then 
$R_0\lrpar{P}$ and $R_1\lrpar{P}$ are both fully elliptic and have the same families index.
\end{proposition}

\begin{proof}
The invertibility of $\sigma\lrpar{P}$ allows the construction of a parametrix \newline $Q_0 \in \Psi^{-m}_{\eps\Phi}\lrpar{M/B;E_-, E_+}$ with residues
\begin{equation*}
	\Id-Q_0P \in \Psi^{-\infty}_{\eps\Phi}\lrpar{M/B;E_+}, \phantom{xxx}
	\Id-PQ_0 \in \Psi^{-\infty}_{\eps\Phi}\lrpar{M/B;E_-}.
\end{equation*}
Each invertible normal operator permits this to be refined to a parametrix with residues vanishing to infinite order at the corresponding boundary face. Thus the hypothesis of the proposition permit us to find $Q \in \Psi^{-m}_{\eps\Phi}\lrpar{M/B;\EE_{-}}$ with residues
\begin{equation*}
	\Id-QP \in x^{\infty}\Psi^{-\infty}_{\eps\Phi}\lrpar{M/B;E_+}, \phantom{xxx}
	\Id-PQ \in x^{\infty}\Psi^{-\infty}_{\eps\Phi}\lrpar{M/B;E_-}.
\end{equation*}
The induced equations for 
$R_0\lrpar{P}$ and $R_1\lrpar{P}$ 
show that these operators are fully elliptic hence Fredholm.

The residues above are families of compact operators over $\lrspar{0,1}$ in the intersection of the fibered cusp and fibered boundary calculi, hence it is possible to stabilize. That is, we can find an operator 
\begin{equation*}
	A \in x^{\infty}\Psi^{-\infty}_{\eps\Phi}\lrpar{M/B;\EE} 
	= \CI\lrpar{ \lrspar{0,1}_{\eps}, \dot{\Psi}^{-\infty}\lrpar{M/B;\EE} }
\end{equation*}
such that the null space of $P+A$ is in $\dot{\mathcal{C}}^{\infty}\lrpar{M/B;E_+}$ and forms a trivial smooth bundle over $\lrspar{0,1}\times B$ (see \cite[Lemma 1.1]{Melrose-Rochon}). It follows that the index bundles for 
$R_0\lrpar{P}$ and $R_1\lrpar{P}$ coincide in $K\lrpar{B}$.
\end{proof}

\begin{theorem} \label{AdMap}
The adiabatic construction above induces a natural map
\begin{equation*}
	\cK^*_{\fC{\Phi}}\lrpar{\phi} \to
	\cK^*_{\fB{\Phi}}(\phi),
\end{equation*}
which restricts to 
\begin{equation*}
	\cK^*_{\fC{\Phi},-\infty}\lrpar{\phi} \to 
	\cK^*_{\fB{\Phi},-\infty}(\phi),	
\end{equation*}
and fits into the commutative diagram
\begin{equation*}
	\xymatrix{
	& \cK^*_{\fC{\Phi}}\lrpar{\phi} \ar[ld]_{\mathrm{ind}_a} \ar[d] \ar[rd]^\sigma & \\
	K^*\lrpar{B} & \cK^*_{\fB{\Phi}}\lrpar{\phi} \ar[l]_{\mathrm{ind}_a} \ar[r]^(.4)\sigma
	& K_{c}\lrpar{T^*\lrpar{X/B}} }
\end{equation*}
\end{theorem}

\begin{proof}
Given a fibered cusp pseudodifferential operator, $P \in \Psi^0_{\fC{\Phi}}\lrpar{M/B;\EE}$ we can construct an element of $\Psi^0_{\eps\phi}\lrpar{M/B;\EE}$ by following the construction in \cite[Proposition 8]{Melrose-Rochon}.
First the kernel of $P$ gives us the normal operator of a putative $P\lrpar{\eps}$ at the (lift of) $\eps=0$. This fixes the boundary value for the face $\df{F}$, which as we have just seen is fibered over the lifted variable 
\begin{equation*}
	\tau = \frac{x-\eps}{x+\eps} \in \lrspar{-1,1},
\end{equation*}
with fiber given by the front face in the fibered cusp calculus. Knowing the value at $\tau=1$, we extend to the rest of $\df{F}$ to be independent of $\tau$.
This fixes the kernel at the boundary of $\df{B}$ and we 
simply extend it as a conormal distribution to the diagonal and vanishing to infinite order at the other boundary faces.
Having defined a kernel consistently on the boundary of $M^2_{\eps\phi}$ we can choose an extension which is conormal to the diagonal in the interior, thus obtaining an element $P\lrpar{\eps} \in \Psi^0_{\eps\Phi}\lrpar{M/B;\EE}$.

If $P$ is Fredholm then the normal operators $\sigma\lrpar{P\lrpar{\eps}}$ and $N_{\df{F}}\lrpar{P\lrpar{\eps}}$ are invertible as is $N_{\df{B}}\lrpar{P\lrpar{\eps}}$ for small enough $\eps$, say $\eps \in \lrspar{0,\delta}$. Of course we can arrange for $N_{\df{B}}\lrpar{P\lrpar{\eps}}$ to be invertible for all $\eps$, for instance by rescaling $\lrspar{0,\delta}$ to $\lrspar{0,1}$ or by perturbing $P\lrpar{\eps}$ with a family $A\lrpar{\eps}$ such that $N_{\df{B}}\lrpar{A\lrpar{\eps}}$ is a finite rank smoothing perturbation vanishing at $\eps =0$ and making the family $N_{\df{B}}\lrpar{P\lrpar{\eps}}$ invertible (as in \cite[Remark 1.2]{Melrose-Rochon}).

The invertibility of these normal operators allows us to apply Proposition \ref{SameIndex} and conclude that $R_1\lrpar{P\lrpar{E}}$ is fully elliptic.
Thus, in terms of the maps
\begin{equation*}
	\xymatrix{ & \Psi_{\eps\Phi}^0(M/B; \EE) \ar[ld]^{R_0}  \ar[rd]_{R_1} & \\ \Psi^0_{\fC{\Phi}}(M/B;\EE) & & \Psi^0_{\fB{\Phi}}(M/B;\EE) }
\end{equation*}
we see that we can lift each fully elliptic family $P$ in the $\fC\Phi$ calculus via $R_0$ to a family $P(\eps)$ in the adiabatic calculus such that $R_1(P(\eps))$ is a fully elliptic family in the $\fB\Phi$ calculus.
Notice that the choice of $\delta$ and extension to $\eps=1$ do not change the homotopy class of $R_1(P(\eps))$. Furthermore, we can always extend a homotopy of $P$ to a homotopy of $P\lrpar{\eps}$ as well as the other equivalence relations defining the K-groups (stabilization and bundle isomorphism), so we have a map
\begin{equation*}
	\cK^*_{\fC{\Phi}}\lrpar{\phi} \ni \lrspar{P} \xrightarrow{\ad}
	\lrspar{R_1\lrpar{P\lrpar{\eps}}} \in \cK^*_{\fB{\Phi}}\lrpar{\phi},
\end{equation*}
well-defined independently of choices.

From the construction it is clear that 
perturbations of the identity by a smoothing operator are preserved as is the interior symbol, 
and from Proposition \ref{SameIndex} so is the families index.
\end{proof}

\section{Fredholm perturbations of the identity} \label{sec:FredRes}

In this section, we prove that the map
\begin{equation*}
	\cK^*_{\fC{\Phi},-\infty}\lrpar{\phi} \to 
	\cK^*_{\fB{\Phi},-\infty}\lrpar{\phi},	
\end{equation*}
induced by the adiabatic construction in Theorem \ref{AdMap} is an isomorphism.

We first recall the excision argument from \cite{Melrose-Rochon} and reduce to the case $\Phi=\Id$, i.e., the scattering and zero calculi.
We start by replacing $M$ with a collar neighborhood of its boundary, $\pa M \times \lrspar{0,1}_x$.
The K-theory groups we are considering are made up of equivalence classes of normal operators, so there is no loss in restricting attention to the quantizations that are supported in this neighborhood, and reduce to the identity at $x=1$.

The fibration extends off the boundary into the entire neighborhood $\pa M \times \lrspar{0,1}_x$, so we have the simpler situation
\begin{equation*}
	\xymatrix { 
	Z \times \lrspar{0,1} \ar@{-}[r] & \pa X \times \lrspar{0,1} \ar@{-}[r] \ar[d] 
		& \pa M \times \lrspar{0,1} \ar[d]\\
	& Y \times \lrspar{0,1} \ar@{-}[r] & D \times \lrspar{0,1} \ar[d]^{\wt{\phi}} \\
	& & B}
\end{equation*}
where furthermore we have the following reduction.

\begin{lemma}\label{ExcisionLemma}
There are `excision' isomorphisms
\begin{equation*}
	\cK^*_{\fC{\Phi},-\infty}\lrpar{\phi} \cong
	\cK^*_{\mathrm{sc},-\infty}(\wt{\phi})
	\text{ and }
	\cK^*_{\fB{\Phi},-\infty}\lrpar{\phi} \cong
	\cK^*_{0,-\infty}(\wt{\phi})
\end{equation*}
\end{lemma}

\begin{proof}
Recall \cite[Lemma 6.3]{Rochon}, \cite[proof of Prop. 3.1]{Melrose-Rochon} that given
\begin{equation*}
	b \in \Psi^{-\infty}_{\Phi\text{-sus}}\lrpar{\pa M/D;E}
\end{equation*}
such that $\Id + b$ is invertible, we can assume that $b$ acts on a finite-rank sub-bundle $W$ of $\CI\lrpar{\pa M/D}$ pulled-back to $T^*\lrpar{D/B}\times \RR$ and then think of $b$ as the boundary family of a family in $\Psi_{\fC{\Phi}}^{-\infty}\lrpar{M/B;E}$ or in $\Psi_{\mathrm{sc}}^{-\infty}\lrpar{D \times \lrspar{0,1}/B;W}$.
Doing the same to fibered boundary operators we obtain the isomorphisms above.
\end{proof}

Hence, it suffices to show that the adiabatic homomorphism between $\cK^*_{\mathrm{sc},-\infty}(\wt{\phi})$ and $\cK^*_{0,-\infty}(\wt{\phi})$ is an isomorphism. Both the kernel of an operator in $\Psi_{\mathrm{sc}}^{-\infty}\lrpar{M;E}$ and the kernel of an operator in $\Psi_0^{-\infty}\lrpar{M;E}$ are, near their respective front faces, translation invariant in the `horizontal' directions. We can apply the Fourier transform in these directions and restrict to the front face, and obtain in either case an element of 
\begin{equation}\label{InftyKernel}
	\Sc\lrpar{T^*\partial M} \wh{\otimes}_\pi \dot{C}^{\infty}\lrpar{\lrspar{-1,1}} \wh{\otimes}_\pi \Hom\lrpar{E}.
\end{equation}
In the scattering case this is easily recognized as defining a class in \newline
$K^{-1}_c\lrpar{T^*\pa M\times \RR} \cong K^{-2}_c\lrpar{T^*\pa M}$ and respecting the product structure.
We show in the following theorem that the zero case defines the same class in $K^0\lrpar{T^*\pa M}$. 
The coincidence of the classes defined by the commutative (scattering) product with that defined by the non-commutative (zero) product is a manifestation of Bott periodicity (cf. \cite[Proposition 9.9]{HHH}).

\begin{theorem}\label{Thm3Sc}
The map $\cK^*_{\mathrm{sc},-\infty}(\wt{\phi}) \to \cK^*_{0,-\infty}(\wt{\phi})$ induced by the adiabatic calculus is an isomorphism and fits into the commutative diagram
\begin{equation*}
	\xymatrix{ 
	\cK_{\mathrm{sc}, -\infty}^*\lrpar{\phi} \ar[r]^{\cong} \ar[d] 
		& K^*_c\lrpar{T^*\pa M/B} \\
	\cK_{0, -\infty}^*\lrpar{\phi} \ar[ur]^{\cong} }
\end{equation*}
\end{theorem}

Before proving this theorem, we quickly review the relevant descriptions of the topological K-theory groups from \cite{Melrose-Rochon0}.
Recall that for any manifold $X$ we can describe $K^1\lrpar{X}$ as stable homotopy classes of maps
\begin{equation*}
	K^1\lrpar{X} = \lim_{\to} \lrspar{X; \mathrm{GL}\lrpar{N} },
\end{equation*}
so also \cite[Proposition 3]{Melrose-Rochon0},
\begin{equation}\label{K1def}
	K^1\lrpar{X} = \lrspar{X; \dot{G}^{-\infty}\lrpar{V;E} },
\end{equation}
where $V$ is any compact manifold with boundary and 
\begin{equation*}
	\dot G^{-\infty}(V;E)
	= \lrbrac{ \Id + A : A \in \dot\Psi^{-\infty}(X;E)
	 \text{ and } (\Id + A) \text{ is invertible} }.
\end{equation*}
Similarly, we can define $K^0(X)$ from $K^1(X)$ by suspension, i.e., 
\begin{equation*}
	K^0\lrpar{X} 
	= \lim_{\to} \lrspar{X; C^{\infty}\lrpar{ \lrpar{\mathbb{S}^1, 1}; \lrpar{\mathrm{GL}\lrpar{N}, \Id}}}
\end{equation*}
and this can also be thought of as \cite[Proposition 4]{Melrose-Rochon0}
\begin{equation}\label{K0def}
	K^0\lrpar{X} 
	= \lrspar{X; G^{-\infty}_{\sus}\lrpar{U;E}},
\end{equation}
with $U$ any closed manifold of positive dimension and, using $\cS$ to denote Schwartz functions,
\begin{equation*}
	G^{-\infty}_{\sus}\lrpar{U;E}
	= \lrbrac{ \Id + A : A \in \cS( \bbR; \Psi^{-\infty}(U;E) )
	 \text{ and } (\Id + A) \text{ is invertible} }.
\end{equation*}
For non-compact manifolds the same definitions apply, but the families of operators are required to be equal to the identity outside a compact set. Homotopies are also required to be the identity outside a compact set, however, in either case, the compact set is not fixed.

\begin{proof}[Proof (of Theorem \ref{Thm3Sc})]
Recall that in \cite{Albin-Melrose} it is shown that the groups $\cK_{0, -\infty}^q(\phi)$ are canonically isomorphic to the groups $K_c^q(T^*\pa M/B)$.
So, for each $\eps_0 >0$, we have an isomorphism into 
$K^0_c\lrpar{T^*\pa M/B}$
consistent with the product on the $\eps =\eps_0$ slice of the adiabatic calculus.
The smoothness of the product in the adiabatic calculus shows that, for any
Fredholm operator $\Id + A \in \Psi_{\eps\phi}^{-\infty}$,
we obtain for each $\eps_0 >0$ the same class in $K^0_c\lrpar{T^*\pa M}$.

Thus we have two homomorphisms from 
$\cK^*_{\mathrm{sc},-\infty}(\phi)$ into $K^0_c\lrpar{T^*\pa M}$. 
The first is the isomorphism given by identifying \eqref{InftyKernel} with an element of 
$K^{-1}_c\lrpar{T^*\pa M\times \RR} \cong K^{-2}_c\lrpar{T^*\pa M}$. The second comes from extending the given Fredholm scattering operator into a Fredholm adiabatic operator and applying the isomorphism of \cite{Albin-Melrose}.
The first can be represented by a map
\begin{equation}\begin{split}\label{ComCase}
	\lrspar{ T^*\pa M \times \RR ; \GL_N\lrpar{\CC}}
	&= \lrspar{S^*\pa M \times \RR_+ \times \RR; \GL_N\lrpar{\CC}}_0\\
	&= \lrspar{S^*\pa M ; \Sc\lrpar{\RR_+\times\RR}\otimes\GL_N\lrpar{\CC}}_0
\end{split}\end{equation}
with the zero subscript denoting that the restriction to zero in the $\RR_+$ factor depends only on the base variable in $S^*\pa M$. The second homomorphism comes from quantizing the product in \eqref{ComCase} into the b,c-calculus with a parameter $\eps$:
\begin{equation*}
	\lrspar{S^*\pa M ; \Sc\lrpar{\RR_+\times\RR}\otimes\GL_N\lrpar{\CC}}_0
	\xrightarrow{q_{\eps}}
	\lrspar{S^*\pa M ; \Psi^{-\infty, -\infty}_{\text{b,c}}\lrpar{\lrspar{0,1};\CC^N}}_0.
\end{equation*}
When $\eps = 0$ we have the commutative product and for positive $\eps$, the non-commutative (cf. \cite[Proposition 9.9, Theorem 9.5]{HHH}).
Hence the isomorphism between 
$\cK^0_{0,-\infty}(\phi)$ and $K^0_c\lrpar{T^*\pa M}$ is taken into the isomorphism between 
$\cK^0_{\mathrm{sc},-\infty}(\phi)$ and $K^0_c\lrpar{T^*\pa M}$ by the adiabatic limit.
\end{proof}

\section{Six term exact sequence} \label{sec:SixTerm}

The Atiyah-Singer Index Theorem induces natural maps 
\begin{equation*}
	K_c^q(T^*M/B) \ni [\sigma] \xrightarrow{I_q} [\ind_{AS}r_{\pa M}^*\sigma] \in K_c^{q+1}(T^*D/B)
\end{equation*}
and, together with the identifications $K_c^{q+1}(T^*D/B) \cong \cK_{\fB{\Phi}, -\infty}^{q+1}(\phi)$, these are maps $K_c^q(T^*M/B) \to \cK_{\fB{\Phi},-\infty}^{q+1}(\phi)$.
Just as in \cite[Thm. 6.3]{Melrose-Rochon}, they fit into a diagram
\begin{equation}\label{SixTermPhiB}
\xymatrix  {
 \cK^0_{\fB{\Phi},-\infty}\lrpar{\phi} \ar[r]^{i_0} & 
\cK^0_{\fB{\Phi}}\lrpar{\phi} \ar[r]^{\sigma_0} &
K^0_c\lrpar{T^*M/B} \ar[d]^{I_0} \\
 K^1_c\lrpar{T^*M/B} \ar[u]^{I_1} &
\cK^1_{\fB{\Phi}}\lrpar{\phi} \ar[l]^{\sigma_1} &
\cK^1_{\fB{\Phi},-\infty}\lrpar{\phi} \ar[l]^{i_1}  }
\end{equation}
with the maps $i_*$ and $\sigma_*$ induced by inclusion and the principal interior symbol, which we now show is exact.

\begin{proposition}
The diagram \eqref{SixTermPhiB} is exact.
\end{proposition}

\begin{proof}
While exactness at $\cK^*_{\fB{\Phi}}(\phi)$ follows directly from the definitions, to establish exactness elsewhere we give a different description of the maps $I_q$.

Given a class $[\sigma_t] \in K^1_c(T^*M/B)$, it is shown in \cite[Proposition 6.2]{Melrose-Rochon} that there is a family of operators
\begin{equation*}
\begin{gathered}
	P_t \in \Psi^0_{\fC{\Phi}}(M/B; \bbC^N) \text{ and Fredholm, such that} \\
	\sigma(P_t) = \sigma_t, \quad
	P_0 = \Id, \quad
	P_1 \in ( \Id + \Psi_{\fC{\Phi}}^{-\infty}(M/B; \bbC^N) ) \text{ and Fredholm}
\end{gathered}
\end{equation*}
and the identification $\cK_{\fC{\Phi}, -\infty}^{0}(\phi) \cong K_c^{0}(T^*D/B)$ takes $[P_1]$ to $ [\ind_{AS}r_{\pa M}^*\sigma]$.
As in the proof of Theorem \ref{AdMap}, we can lift the family $P_t$ to a family of operators in the adiabatic calculus $P_t(\eps)$ such that $\wt P_t = R_1(P_t(\eps))$ satisfies
\begin{equation*}
\begin{gathered}
	\wt P_t \in \Psi^0_{\fB{\Phi}}(M/B; \bbC^N) \text{ and Fredholm, such that} \\
	\sigma(\wt P_t) = \sigma_t, \quad
	\wt P_0 = \Id, \quad
	\wt P_1 \in ( \Id + \Psi_{\fB{\Phi}}^{-\infty}(M/B; \bbC^N) ) \text{ and Fredholm}
\end{gathered}
\end{equation*}
and, from Theorem \ref{Thm3Sc},  the identification $\cK_{\fB{\Phi}, -\infty}(\phi) \cong K_c(T^*D/B)$ takes $[\wt P_1]$ to $ [\ind_{AS}r_{\pa M}^*\sigma]$.
Thus we can think of $I_1$ as being $[\sigma_t] \mapsto [\wt P_1]$.

An element $[\sigma_t]$ is in the null space of $I_1$ if $\wt P_1 \sim \Id$, i.e., precisely when $[\sigma_t]$ is in the image of the map $K^1_{\fB{\Phi}}(\phi) \xrightarrow{\sigma_1} K_c^1(T^*M/B)$.
On the other hand the homotopy $\wt P_t$ shows that $[\wt P_1]$ is automatically trivial as an element $K^0_{\fB{\Phi}}(\phi)$, that is, $\Im I_1 \subseteq \nul i_0$,
and conversely, if $[Q] \in \nul i_0$, there is (after stabilization) a homotopy of operators from the identity to $Q$, and the symbol of this homotopy defines an element in $K_c^1(T^*M/B)$ that is sent by $I_1$ to $[Q]$.
Thus we see that the diagram is exact at $K_c^1(T^*M/B)$ and at $\cK_{\fB{\Phi}}^1(\phi)$. 
Just as in \cite[Theorem 6.3]{Melrose-Rochon} the same arguments prove exactness at $K_c^1(T^*M/B)$ and $\cK_{\fB{\Phi}}^1(\phi)$ using Bott periodicity. 
\end{proof}

As explained in the proof of this proposition, the diagrams
\begin{equation*}
	\xymatrix{ K^q_c(T^*M/B) \ar[r]^{I_q} \ar[rd]^{I_q} & \cK^{q+1}_{\fC{\Phi}}(\phi) \ar[d]^{\ad} \\ & \cK^{q+1}_{\fB{\Phi}}(\phi) }
\end{equation*}
are commutative. Hence the six term exact sequences for the smooth K-theory of $\fC{\Phi}$  and $\fB{\Phi}$ operators make up a commutative diagram:
\begin{equation*}
\xymatrix  @R=17pt @C=17pt {
\cK^0_{\fC{\Phi},-\infty}\lrpar{\phi} \ar[rr] \ar[rd]^{\ad} & & 
\cK^0_{\fC{\Phi}}\lrpar{\phi} \ar[rr] \ar[d]^{\ad} & & 
K^0_c\lrpar{T^*M/B} \ar@{->}[ddd] \ar@{<->}[ld]^= \\
& \cK^0_{\fB{\Phi},-\infty}\lrpar{\phi} \ar[r] & 
\cK^0_{\fB{\Phi}}\lrpar{\phi} \ar[r] &
K^0_c\lrpar{T^*M/B} \ar[d] & \\
& K^1_c\lrpar{T^*M/B} \ar[u] &
\cK^1_{\fB{\Phi}}\lrpar{\phi} \ar[l] &
\cK^1_{\fB{\Phi},-\infty}\lrpar{\phi} \ar[l] & \\
K^1_c\lrpar{T^*M} \ar@{->}[uuu] \ar@{<->}[ur]^= & &
\cK^1_{\fC{\Phi}}\lrpar{\phi} \ar[ll] \ar[u]^{\ad} & &
\cK^1_{\fC{\Phi},-\infty}\lrpar{\phi} \ar[ll] \ar[lu]^{\ad} \\}
\end{equation*}
and, since we have already shown that $\cK^{q}_{\fC{\Phi}}(\phi) \xrightarrow{\ad} \cK^{q}_{\fB{\Phi}}(\phi)$ is an isomorphism, the five-lemma implies the following.

\begin{theorem}
The adiabatic calculus induces an isomorphism of the smooth K-theory groups of the $\fC{\Phi}$ calculus and those of the $\fB{\Phi}$ calculus,
\begin{equation*}
	\cK^{q}_{\fC{\Phi}}(\phi) \xrightarrow[\cong]{\ad} \cK^{q}_{\fB{\Phi}}(\phi).
\end{equation*}
\end{theorem}

\appendix
\section{Triple adiabatic space} 

In this section we prove the composition formula for the adiabatic calculus. 
The kernel of the composition is given by an appropriate interpretation of 
\begin{equation*}
	\cK_{A\circ B} \lrpar{ \zeta, \zeta', \eps} = 
	\int_{M^2_{\eps\phi}} \cK_A \lrpar{ \zeta, \zeta'', \eps} \cK_B \lrpar{\zeta'', \zeta', \eps}.
\end{equation*}
Our method is geometric, we construct a space $M^3_{\eps\phi}$ with three b-fibrations `first', `second', and `composite' down to $M^2_{\eps\phi}$. The kernel of the composite is then correctly given by
\begin{equation*}
	\lrpar{\pi_C}_*\lrpar{\pi_F^*A \cdot \pi_S^*B}.
\end{equation*}

The basic triple space is $M^3\times[0,\epsilon _0]$ where $\epsilon _0>0$
and $M$ is a compact manifold with boundary and a specified fibration $\phi
:\pa M\longrightarrow Y$ of the boundary. We first have to perform the
`boundary adiabatic' blow ups to be able to map back to the single and
double spaces
\begin{equation}
\begin{gathered}
M^3_{(1)}=[M^3\times[0,\epsilon _0],A_T,A_F,A_C,A_S,A_R,A_M,A_L],\\
A_T=(\pa M)^3\times\{0\},\
A_F=M\times\pa M\times\pa M\times\{0\},\\
A_S=\pa M\times\pa M\times M\times\{0\},\
A_C=\pa M\times M\times\pa M\times\{0\},\\
A_R=M\times M\times\pa M\times\{0\},\
A_M=M\times\pa M\times M\times\{0\},\\
A_L=\pa M\times M\times M\times\{0\}.
\end{gathered}
\label{Afbct.1}\end{equation}

Then consider the lifts of the triple and double fiber diagonals, lying
within the lifts of the faces $(\pa M)^3\times[0,\epsilon _0]$ and 
$(\pa M)^2\times M\times[0,\epsilon _0]$ and its cyclic images. We can
denote these as $\df{B}_T,$ the triple boundary fibred diagonal and $\df{B}_S,$
$\df{B}_C$ and $\df{B}_F,$ the double fibered diagonals. These are all $p$-submanifolds,
meeting in the standard way for diagonals so we may define 
\begin{equation}
M^3_{(2)}=[M^3_{(1)},\df{B}_T,\df{B}_S,\df{B}_C,\df{B}_F].
\label{Afbct.2}\end{equation}

The double and triple fibered-cusp diagonal faces are not $p$-submanifolds in
$M_{(1)}$ but lift to be $p$-submanifolds, which we denote $\df{F}_T,$ $\df{F}'_F,$
$\df{F}'_C,$ $\df{F}'_S$ and $\df{F}''_F,$ $\df{F}''_C$ and $\df{F}''_F.$ Here, $\df{F}'_O$ is the lift
of the corresponding double diagonal in the face produced by the blow-up of
$A_O,$ $O=S,C,F$ and $\df{F}''_O$ is the lift of the intersection of this
submanifold under the blow-up of $A_T.$ The intersection properties of
these submanifolds is essentially the same as for the fibered-cusp (or
indeed the cusp) setting itself. Thus we complete the definition by blowing
up in `the usual' order 
\begin{equation}
M^3_{\epsilon \phi }=[M^3_{(2)},\df{F}_T,\df{F}''_S,\df{F}''_C,\df{F}''_F,\df{F}'_S,\df{F}'_C,\df{F}'_F].
\label{Afbct.3}\end{equation}

\begin{proposition}\label{Afbct.4} Each of the projections, dropping one or
two factors of $M,$ lift to b-fibrations 
\begin{equation}
\begin{gathered}
\xymatrix@1{
{\begin{pmatrix}
\pi_F
\\
\pi_C
\\
\pi_S
\end{pmatrix}}
:
M^3_{\epsilon\phi}\ar[r]\ar@<1ex>[r]\ar@<-1ex>[r]
&
M^2_{\epsilon\phi}
},\\
\xymatrix@1{
{\begin{pmatrix}
\pi_L
\\
\pi_M
\\
\pi_R
\end{pmatrix}}
:
M^3_{\epsilon\phi}
\ar[r]\ar@<1ex>[r]\ar@<-1ex>[r]
&
M_{\epsilon}.}
\end{gathered}
\label{Afbct.5}\end{equation}
\end{proposition}

\begin{proof} This is the usual commutativity of blow ups
  argument. Rearranging the `fibred-cusp' faces works just as for the
  (fibered-) cusp calculus. After these triple and two double faces have
  been blown down, the corresponding fibered-boundary faces can be blown
  down, as can the pure adiabatic faces.
\end{proof}

\begin{proposition}\label{AdComposition}
\begin{equation*}
	\Psi^k_{\eps\phi}\lrpar{M;G,H} \circ 
	\Psi^{k'}_{\eps\phi}\lrpar{M;E,G} \subset
	\Psi^{k+k'}_{\eps\phi}\lrpar{M;E,H}
\end{equation*}
\end{proposition}

\begin{proof}
This follows from Proposition \ref{Afbct.4} via the push-forward and pull-back theorems of \cite{Corners}.
Indeed, given two operators (acting on half-densities for simplicity) $A$ and $B$ the kernel of their composition is given by
\begin{equation*}
	A\circ B = \lrpar{\pi_C}_*\lrpar{\pi_S^*A\cdot \pi_F^*B}.
\end{equation*}
As the maps are b-fibrations and transversal to the diagonals this yields a distribution conormal to the diagonal of the appropriate degree. 
Since the kernels of $A$ and $B$ vanish to infinite order at every face not meeting the diagonal, the product of their lifts will vanish to infinite order at every face not meeting the triple diagonal (as this is the intersection of the lift of the two diagonals). Hence only the faces coming from $\diag\lrpar{M^3} \times \{0\}$, and the blow-ups of $\df{B}_T$ and $\df{F}_T$ potentially contribute to the push-forward. It follows that the resulting distribution is smooth down to the each of the boundary faces.
\end{proof}

\bibliography{fresmb}
\bibliographystyle{amsplain}

\end{document}